\documentclass{amsart}

\usepackage{graphicx}
\usepackage{amsmath}
\usepackage{amsfonts}
\usepackage{amssymb}
\usepackage{xypic}
\usepackage[all]{xy}
\usepackage{array}
\usepackage{enumerate}
\usepackage{url}

\newtheorem{thm}{Theorem}[section]

\newtheorem{cor}[thm]{Corollary}

\newtheorem{exmp}[thm]{Example}

\newtheorem{lem}[thm]{Lemma}

\newtheorem{prop}[thm]{Proposition}

\newcommand{\DD}{\mathcal{D}}

\newcommand{\LL}{\mathcal{L}}
\newcommand{\RR}{\mathcal{R}}

\newcommand{\Drel}{\mathbin{\DD}}

\newcommand{\suparrow}{{\uparrow}}
\newcommand{\duparrow}{{\downarrow}}

\newcommand{\dd}{\mathbin{\setminus\!\setminus}}

\newcommand{\restricted}[2]{#1{\mid}_{#2}}

\def\dom{\mathrm{dom}}

\def\id{\mathrm{id}}

\begin{document}

%\begin{frontmatter}   %%  Title, information about author, abstract, etc.

\title{On Skew Heyting Algebras}           % title of the paper
{}                 % footnote on the title -- empty if not required

\author{Karin Cvetko-Vah}            % First author name
%{University of Ljubljana, Faculty of Mathematics and Physics, Jadranska 19, Ljubljana, Slovenia}    % Affiliation and address
%{karin.cvetko@fmf.uni-lj.si}                     % E-mail address
{}          % Footnote on the first author (grant number, thanks, 
                                         % web page, etc.) -- empty in not required

%\authordata{Second B. Author}            % Second author
%{Institution, Address, City, Country} 
%{email@address2.com}
%{}                                       % No footnote!
%
%\authordata{Third. C. Author}            % Third author
%{Institution, Address, City, Country}
%{email@address3.com}
%{Footnote on the third author.}

%\keywords{Skew lattices, Heyting algebras, non-commutative algebra, intuitionistic logic.}               % Keywords
\maketitle
%\msc{06F35, 03G27}                       % Math. Subj. Class. codes

\begin{abstract}
In the present paper we  generalize the notion of a Heyting algebra to the non-commutative setting and hence introduce what we believe to be the proper notion of the implication in skew lattices. We list several examples of skew Heyting algebras, including Heyting algebras, dual skew Boolean algebras, conormal skew chains and algebras of partial maps with poset domains.\end{abstract}

%\end{frontmatter}   %% End of the front matter

%% Your article

\section{Introduction}

Non-commutative generalizations of  lattices  were introduced by Jordan  \cite{jordan} in 1949. The current approach to such objects began with  Leech's 1989 paper on skew lattices \cite{L1}.  Similarly, skew Boolean algebras are 
non-commutative generalizations of  Boolean algebras. In 1936 Stone proved that  each Boolean algebra can be embedded into a field of sets \cite{Sto1936}. Likewise, Leech showed in  \cite{L2, L4} that each 
right-handed skew Boolean algebra can be embedded into a generic  skew Boolean  algebra of partial functions from a given set to the codomain $\{0,1\}$. Bignall and Leech \cite{bl} showed that skew Boolean algebras play a central role in the study of discriminator varieties.

 Though not  using  categorical language, Stone essentially proved   in \cite{Sto1936} that the category of Boolean algebras and homomorphisms is dual to the category of Boolean topological spaces and continuous maps.
Generalizations of this result within the commutative setting yield  Priestley duality \cite{P, P1} between bounded distributive lattices and Priestley spaces, and Esakia duality \cite{esakia} between Heyting algebras and Esakia spaces. (See \cite{nick} for details.) In a recent paper \cite{Gehrke} on Esakia's work, Gehrke showed that Heyting algebras
may be understood as those distributive lattices for which the embedding into their
Booleanisation has a right adjoint. A recent line of research  generalized the results of Stone and Priestley  to the 
non-commutative setting. By results  in \cite{BCV} and \cite{K}, any skew Boolean algebra is dual to a sheaf of rectangular bands over a locally-compact Boolean space. A further generalization  given in \cite{skew-priestley} showed that any  strongly distributive skew lattice (as defined below) is dual to a sheaf (of rectangular bands) over a locally compact Priestley space.

 While Boolean algebras provide algebraic models of classical logic, Heyting algebras provide algebraic models of intuitionistic logic.  In the present paper we introduce the notion of a skew Heyting algebra. In passing to the 
non-commutative setting one needs to sacrifice either the top or the bottom of the algebra in order not to end up in the commutative setting. In the previous papers \cite{BCV}, \cite{K} and \cite{skew-priestley} algebras with bottoms were considered, and hence  the notion of distributivity was generalized to the notion of  so-called strong distributivity. If one tried to define an implication operation in the setting of   strongly distributive skew lattices with a bottom as a right adjoint to  conjunction, that would force the skew lattice to  also  possess a top and hence  be commutative,  resulting in a usual Heyting algebra.
In order to define  implication in the skew lattice setting we consider  the $\lor - \land$ duals of strongly distributive skew lattices with a bottom, namely, the  co-strongly distributive skew lattices with a top. That is
 not surprising as a top plays a crucial role in logic. The category of co-strongly distributive skew lattices with a top is, of course, isomorphic to the category of strongly distributive skew lattices with a bottom. In choosing  co-strongly distributive skew lattices with a top we  follow the path paved by Bignall and Spinks in \cite{bs1997}, and by Spinks and Veroff in \cite{SV2006} where dual skew Boolean algebras were introduced.
For further reading on  implications in skew Boolean algebras and their algebraic duals, see \cite{bsv2014}. 

After reviewing some preliminary definitions and concepts in Section 2, in the next
section we introduce the notion of a skew Heyting algebra, prove that such algebras form
a variety and show that the maximal lattice image of a skew Heyting algebra is a
generalized Heyting algebra (possibly without a bottom). Indeed, a co-strongly
distributive skew lattice with a top is the reduct of a skew Heyting
algebra, if and only if its maximal lattice image forms a generalized
Heyting algebra. (See Theorem 3.5.) This leads to a number of useful
corollaries and examples. We finish with Section 4 where we explore the
consequences of our results to  duality theory, and describe how skew Heyting
algebras correspond to  sheaves over local Esakia spaces.

\section{Preliminaries}

A \emph{skew lattice} is an algebra $\mathbf S=(S;\land, \lor)$ of type $(2,2)$ such that
$\wedge$ and $\vee$ are both idempotent and associative and they satisfy the following absorption laws:
\[%
x\land (x\lor y)=x=x\lor (x\land y) \text{ and } (x\land y)\lor y=y=(x\lor y)\land y.
\]
These identities are collectively equivalent to the pair of equivalences: $x\land y=x\Leftrightarrow x\lor y=y$ and $x\land y=y\Leftrightarrow x\lor y=x$.

On a skew lattice $\mathbf S$ one can define the \emph{natural partial order} by stating that $x\leq y$ if and only if $x\lor y=y=y\lor x$, or equivalentely $x\land y=z=y\land x$, and the \emph{natural preorder} by $x\preceq y$ if and only if $y\lor x\lor y=y$, or equivalentely $x\land y\land x=x$.  \emph{Green's equivalence relation $\mathcal D$} is then defined  by 
\begin{equation}\label{eq:relD}
x\mathcal D y \text{ if and only if } x\preceq y \text{ and } y\preceq x.
\end{equation}
%
%$\mathcal D$ is a congruence on any skew lattice $\mathbf S$, and $\mathbf S/\mathcal D$ is its maximal lattice image. 
%In fact, $x\leq y$ is equivalent to $x\land y=x=y\land x$, while $x\preceq y$ is equivalent to $x\land y\land x=x$. 
%$\mathbf S$ is commutative iff  $\leq $ and $\preceq$ coincide.

 \begin{lem} \label{lemma:costa} (\cite{Costa}). For elements $x$ and $y$ elements of a skew lattice $\mathbf S$ the following are equivalent:
    \begin{itemize}
              \item[(i)] $x\leq y$,
              \item[(ii)] $x\lor y\lor x = y$,
              \item[(iii)] $y\land x\land y=x$.
    \end{itemize}
 \end{lem}

Leech's First Decomposition Theorem for skew lattices states that the relation $\mathcal D$ is a congruence on a skew lattice $\mathbf S$, $\mathbf S/\mathcal D$ is the maximal lattice
image of $\mathbf S$, and each congruence class is a maximal rectangular skew lattice in $S$ \cite{L1}. Rectangular skew lattices are characterized by $x\land y\land z=x\land z$, or equivalentely $x\lor y\lor z= x\lor z$. 
We  denote the $\DD$-class containing an element $x$ by $\DD_x$.

It was proved in \cite{L1} that a skew lattice always forms a \emph{regular band} for either of the operations $\land$, $\lor$, i.e. it satisfies the identities
\[x\land u\land x\land v\land x = x\land u\land v \land x \text{ and }
x\lor u\lor x\lor v\lor x= x\lor u\lor v \lor x.
\]

A \emph{skew lattice with top} is an algebra $(S;\land, \lor,1)$ of type $(2,2,0)$ such that $(S;\land, \lor)$ is a skew lattice and $x\lor 1=1=1\lor x$, or equivalently $x\land 1=x=1\land x$, holds for all $x\in S$. A skew lattice with bottom is defined dually and the bottom, if it exists, is usually denoted by $0$.

Furthermore, a skew lattice is called 
\emph{strongly distributive} if it satisfies the following identities:
\[x \land (y\lor z)=(x\land y) \lor (x\land z) \text{ and } (x\lor y) \land z = (x\land z) \lor (y\land z);
\]
and it is called
\emph{co-strongly distributive}  if it satisfies the identities:
\[x \lor (y\land z)=(x\lor y) \land (x\lor z) \text{ and } (x\land y) \lor z = (x\lor z) \land (y\lor z).
\]

If a skew lattice $\mathbf S$ is either strongly distributive or co-strongly distributive then $\mathbf S$ is \emph{distributive} in that it satisfies the identities
\[x\land (y\lor z)\land x=(x\land y\land x)\lor (x\land z \land x) \text{ and }
x\lor (y\land z)\lor x=(x\lor y\lor x)\land (x\lor z \lor x).
\]

A skew lattice $\mathbf S$ that is jointly strongly distributive and co-strongly distributive  is \emph{binormal}, i.e. $\mathbf S$ factors as a direct product of a lattice $\mathbf L$ and a rectangular skew lattice $\mathbf B$, $\mathbf S\cong\mathbf  L\times\mathbf  B$, with $\mathbf L$  in this case being distributive. (See  \cite{L4} and \cite{schein}.)

Applying duality to a result of Leech \cite{L4}, it follows that a skew lattice $S$ is co-strongly distributive if and only if $\mathbf S$ is jointly:
\begin{itemize} 
   \item \emph{quasi-distributive}: the maximal lattice image $\mathbf S/\mathcal D$ is a distributive lattice,
   \item \emph{symmetric}: $x\land y=y\land x$ if and only if $x\lor y=y\lor x$, and
   \item \emph{conormal}: $x\lor y \lor z\lor w=x\lor z\lor y\lor w$.
\end{itemize}

If a skew lattice is conormal then given any $u\in S$ the set
\[u\suparrow=\{ u\lor x\lor u\, |\, x\in S\}= \{ x\in S\,|\, u\leq x\}
\]
forms a (commutative) lattice for the induced operations $\land$ and $\lor$,
cf. \cite{L4}.

The following lemma is the dual of a well known result in  skew lattice theory.

\begin{lem}\label{lemma:conormal}
Let $\mathbf S$ be a conormal skew lattice and let $A$ and $B$ be $\DD$-classes such that $B\leq A$ holds in the lattice $\mathbf S/\DD$. Given $b\in B$ there exists a unique $a\in A$ such that $b\leq a$.
\end{lem}

\begin{proof}
First the uniqueness. If $a$ and $a'$ both satisfy the desired property then by Lemma \ref{lemma:costa} we have $a=b\lor a\lor b$ and likewise $a'=b\lor a'\lor b.$ Now, using idempotency of $\lor$, conormality  and the fact that $a\Drel a'$ we obtain: 
\begin{multline*}
a=b\lor a \lor b = b\lor a \lor a'\lor a \lor b = \\ 
b\lor a\lor a'\lor b=
b\lor a'\lor a\lor a'\lor b = b\lor a'\lor b =a'. 
\end{multline*}
To prove the existence of $a$ take any $x\in A$ and set $a=b\lor x\lor b$. Then $a\in A$ and using the idempotency of $\lor$ we get:
\[b\lor a\lor b = b\lor (b\lor x\lor b) \lor b= b\lor x\lor b= a
\]
which implies $b\leq a$.
\end{proof}

An important class of strongly distributive skew lattices that have a bottom is the class of skew Boolean algebras where by a \emph{skew Boolean algebra} we mean an algebra $\mathbf S=(S;\land, \lor,\setminus, 0)$ where  $(S;\land, \lor, 0)$ is a strongly distributive skew lattice with bottom $0$, and $\setminus$ is a binary operation on $S$ such that both $(x\land y\land x)\lor (x\setminus y) = x = (x\setminus y)\lor (x\land y\land x)$ and  $(x\land y\land x)\land (x\setminus y) = 0 = (x\setminus y)\land (x\land y\land x)$. Given any element $u$ of a {skew Boolean algebra} $\mathbf S$ the set
\[u\duparrow =\{u\land x\land u \,|\, x\in S\}=\{x\in S \,|\, x\leq u \}
\]
is a Boolean algebra with top $u$ and with $u\setminus x$ being the complement of $u\land x\land u$ in $u\duparrow$.

Recall that a \emph{Heyting algebra} is an algebra $\mathbf H=(H;\land,\lor, \rightarrow, 1, 0)$ such that $(H,\land,\lor, 1, 0)$ is a bounded distributive lattice that satisfies the condition:
\begin{description}
   \item[(HA)] $x\land y\leq z$ iff $x\leq y\rightarrow z$.
 \end{description}
Stated otherwise, $\forall y,z\in H$ the sublattice $\{x\in H\,|\, x\land y\leq z\}$ is nonempty and contains a top element to be denoted by $y\rightarrow z$. Thus, given a bounded distributive lattice $(H; \land, \lor, 1, 0)$, if a binary operation $\rightarrow$ exists that makes $(H; \land, \lor, \rightarrow, 1, 0)$ a Heyting algebra, then it is unique because it is already there implicitly. Indeed, given two isomorphic lattices, if either is the lattice reduct of a Heyting algebra then so is the other, and both are isomorphic as Heyting algebras.
 
Equivalently, (HA) can be replaced by the following set of identities:
\begin{description}
     \item[(H1)] $(x\rightarrow x)=1$,
     \item[(H2)] $x\land (x\rightarrow y)=x\land y$,
     \item[(H3)] $y\land (x\rightarrow y)=y$,
     \item[(H4)] $x\rightarrow (y\land z)=(x\rightarrow y)\land (x\rightarrow z).$
\end{description}

\begin{lem}
In any Heyting algebra, $x\rightarrow y = (x\lor y)\rightarrow y$.
\end{lem}

A \emph{generalized Heyting algebra} is an algebra $\mathbf A=(A;\land,\lor,\rightarrow,  1)$ such that the reduct $(A,\land,\lor,  1)$ is a distributive lattice with top 1, and condition (HA) holds. If it also has a
bottom, it is a Heyting algebra. In general, each upset $u\suparrow$ forms a Heyting algebra. The
identities above also characterize this more general class of algebras, which are often
called \emph{Brouwerian algebras}.

\section{Skew Heyting algebras}

A \emph{skew  Heyting lattice} is an algebra $\mathbf S=(S;\land, \lor , 1)$ of type $(2,2,0)$ such that:
\begin{itemize}
   \item $(S;\land ,\lor, 1) $ is a co-strongly distributive skew lattice with top $1$. Each upset $u\sup$ is thus a bounded distributive lattice.
   \item for any $u\in S$ an operation $\rightarrow_u $ can be defined on  $u\suparrow$ such that $(u\suparrow; \land, \lor , \rightarrow_u, 1, u)$ is a Heyting algebra with top $1$ and bottom $u$.  
\end{itemize} 

Given a skew  Heyting lattice $\mathbf S$, define  $\rightarrow $ on $S$ by setting  
$$x\rightarrow y=(y\lor x\lor y)\rightarrow_y y.$$
A \emph{skew Heyting algebra} is an algebra $\mathbf S =(S;\land,\lor,\rightarrow, 1)$ of type $(2,2,2,0)$ such that $(S;\land,\lor, 1)$ is a  skew  Heyting lattice and $\rightarrow$ is the implication thus induced. A sense of global coherence for $\rightarrow$ on $S$ is given by:

\begin{lem}\label{lemma:impl-well-defined} Let $\mathbf S$ be a skew  Heyting  lattice with $\rightarrow$  as defined above and let $x,y,u\in S$ be such that both $x,y\in u\suparrow$ hold. Then $x\rightarrow y= x\rightarrow_u y$.
\end{lem}

\begin{proof}
As $x$ and $y$ both lie in $u\suparrow$, they commute. By the definition of  $\rightarrow$,  $x\rightarrow y=(x\lor y)\rightarrow_y y\geq y$ by (H3). On the other hand, since $\rightarrow_u$ is the Heyting implication in the Heyting algebra $u\suparrow$ it follows that $x\rightarrow_u y= (x\lor y)\rightarrow_u y\geq y$. Thus $y$, $x\lor y$, $(x\lor y)\rightarrow_y y$ and $(x\lor y)\rightarrow_u y$ all lie iin the Heyting algebra $y\suparrow$. The maximal element characterization of both $(x\lor y)\rightarrow_y y$ and $(x\lor y)\rightarrow_u y$ forces both to agree.
%
%\[
%x\lor y\rightarrow_u y\leq x\lor y\rightarrow_y y \text{ iff } (x\lor y\rightarrow_u y) \land (x\lor y)\leq y \text{ iff } (x\lor y)\land y\leq y,
%\]
%which is true by absorption.
%
%On the other hand, computing in the Heyting algebra $u\suparrow$ we obtain:
%\[
%x\lor y\rightarrow_y y\leq x\lor y\rightarrow_u y \text{ iff } (x\lor y\rightarrow_y y) \land (x\lor y)\leq y \text{ iff } (x\lor y)\land y\leq y,
%\]
%which is again true by absorption.
\end{proof}

We will use the axioms of Heyting  algebras to derive an  axiomatization of {skew Heyting algebras}. The reader should find most  of the axioms of Theorem \ref{theorem:axiomatization} below to be intuitively clear generalizations to the non-commutative case. However, two axioms should be given further explanation. Firstly, the $u$ in axiom (SH4) below appears since upon passing to the non-commutative case, an element that is both below $x$ and $y$ with respect to the  partial order $\leq$ no longer need exist.  (We can have $x\land y\land x\leq x$ but not $x\land y\land x\leq y$  in general.) Similarly, axiom (SH0) is needed since in the non-commutative case it no longer follows from the other axioms, the reason being that in general $x\leq y\lor x\lor y$ need not hold. 

 \begin{thm}\label{theorem:axiomatization} Let  $(S;\land,\lor,\rightarrow, 1)$ be an algebra of type $(2,2,2,0)$ such that $(S;\land,\lor, 1)$ is a co-strongly distributive skew lattice with top $1$. Then $(S;\land,\lor,\rightarrow, 1)$ is a skew Heyting algebra if and only if it  satisfies the following axioms:
\begin{description}
  \item[(SH0)] $x \rightarrow y = (y\lor x\lor y)\rightarrow y$.
  \item[(SH1)] $x\rightarrow x=1$,
  \item[(SH2)] $x\land (x\rightarrow y) \land x = x\land y \land x$, 
  \item[(SH3)] $y\land (x\rightarrow y)=y$ and  $(x\rightarrow y)\land y=y$,
  \item[(SH4)] $x \rightarrow (u\lor (y\land z)\lor u)=
  ( x\rightarrow (u\lor y\lor u)) \land (x \rightarrow (u\lor  z\lor u))$.
\end{description}
\end{thm}

\begin{proof}
Assume that $\mathbf S$ is a skew Heyting algebra.

(SH0). Both $x\rightarrow y$ and  $(y\lor x\lor y)\rightarrow y$ are defined as  $(y\lor x\lor y)\rightarrow_y y$. Thus they are equal.

(SH1). This is true because $1\land x=x$ is true in $x\suparrow$.

(SH2). In $y\suparrow$ (H2) implies  $(y\lor x\lor y)\land ((y\lor x\lor y)\rightarrow_y y)=(y\lor x\lor y)\land y=y.$ Hence
\[
x\land (y\lor x\lor y)\land (x\rightarrow y) \land x=x\land y \land x.
\]
On the other hand,
\[
x\land (y\lor x\lor y)\land (x\rightarrow y) \land x=x\land (y\lor x\lor y)\land x\land (x\rightarrow y) \land x=x\land (x\rightarrow y)\land x,
\]
where we have used the regularity of $\land$ and the fact that $x\preceq y\lor x\lor y$.

(SH3). Both identities hold because $y\land (y\lor x\lor y)=y$ in $y\suparrow$. Thus $x\rightarrow y=(y\lor x\lor y)\rightarrow y\geq y$.

(SH4).
First note that (SH4) is equivalent to 
\begin{itemize}
\item[(SH4')] $(u\lor x \lor u)\rightarrow (u\lor (y\land z)\lor u)=
  ((u\lor x \lor u)\rightarrow (u\lor y\lor u)) \land ((u\lor x \lor u)\rightarrow (u\lor  z\lor u)).$
 \end{itemize} 
Indeed, (SH0) and the conormality of $\lor$ give both
  \[(u\lor x\lor u) \rightarrow (u\lor w\lor u) = (u\lor x\lor w\lor u) \rightarrow (u\lor w\lor u)
  \]
   and  
   \[ x \rightarrow (u\lor w\lor u) = (u\lor x\lor w\lor u) \rightarrow (u\lor w\lor u)
   \]
so that
\[x\rightarrow (u\lor w\lor u)=(u\lor x\lor u)\rightarrow (u\lor w\lor u).
\]
Hence it suffices to prove that (SH4') holds. 

Observe that distributivity implies  
\begin{equation}\label{eq-sh4}
(u\lor y\lor u)\land (u\lor z\lor u)=u\lor (y\land z)\lor u.
\end{equation}
 Since $u\lor x\lor u$, $u\lor y\lor u$, $u\lor z\lor u$ and $u\lor (y\land z)\lor u$ all lie in $u\suparrow$ we can simply compute in $u\suparrow$. Using \eqref{eq-sh4} and  axiom (H4) for Heyting algebras we obtain: 
$(u\lor x \lor u)\rightarrow (u\lor (y\land z)\lor u)= (u\lor x \lor u)\rightarrow ((u\lor y\lor u)\land (u\lor z\lor u))
=
  ((u\lor x \lor u)\rightarrow (u\lor y\lor u))\land ((u\lor x \lor u)\rightarrow (u\lor  z\lor u))$.

To prove the converse assume that (SH0)--(SH4) hold. Given arbitrary $u\in S$ and  $x,y, z\in u\suparrow$ set $x\rightarrow_u y=x\rightarrow y$. Axiom (SH3) implies that $x\rightarrow y\in y\suparrow \subseteq u\suparrow$. Thus the restriction $\rightarrow_u$ of $\rightarrow$ to $u\suparrow$ is well defined. Since $u\suparrow$ is commutative with bottom $u$, axioms (SH1)--(SH4) for $\rightarrow$ respectively simplify to (H1)--(H4) for $\rightarrow_u$, making $\rightarrow_u$ the Heyting implication on $u\suparrow$.  Axiom (SH0) assures that $\rightarrow$ is indeed the skew Heyting implication satisfying $a\rightarrow b=(b\lor a\lor b)\rightarrow_b b,$ for any $a,b\in S$.  
\end{proof}

\begin{cor}
Skew Heyting algebras form a variety.
\end{cor}

In the remainder of the paper, given a skew Heyting algebra  we shall simplify the notation $\rightarrow _u$ to $\rightarrow$ when referring to the Heyting implication in $u\suparrow$. Remarks made about Heyting algebras in Section 2 apply here also. Given a co-strongly distributive skew lattice $(S;\land, \lor,1)$ with a top $1$, if a binary operation $\rightarrow$ exists that makes  $(S;\land, \lor,\rightarrow,1)$ a skew Heyting algebra, then it is unique since it is already there implicitly. Hence, given two isomorphic skew lattices, if either is the reduct of a skew
Heyting algebra, then so is the other and both are isomorphic as skew Heyting algebras.

\begin{prop}\label{lemma:implication-Drel}
The relation $\DD$ defined in \eqref{eq:relD} is a congruence on any skew Heyting algebra.
\end{prop}

\begin{proof}
Let  $(S;\land, \lor,\rightarrow, 1)$ be a skew Heyting algebra. Since $\mathcal D$ is  a congruence for co-strongly distributive skew lattices with a top, we only need to prove   $(a\rightarrow b)\Drel (c\rightarrow d)$ holds for any $a,b,c,d\in S$ satisfying $a\Drel c$ and $b\Drel d$.
Without loss of generality we may assume $b\leq a$ and $d\leq c$. (Otherwise replace $a$ by $b\lor a\lor b$ and $c$ by $d\lor c\lor d$.)

To begin, define a map $\varphi : b\suparrow \to d\suparrow$ by setting $\varphi(x) = d\lor x\lor d$. We claim that $\varphi$ is a lattice isomorphism of  $ (b\suparrow; \land, \lor)$ with $(d\suparrow; \land, \lor)$,  with  inverse  $\psi: d\suparrow \to b\suparrow$ given by $\psi (y)=b\lor y\lor b$. It is easily seen that  $\varphi $ and $\psi$ are inverses of each other. For instance,  
$\psi(\varphi (x)) = b\lor d\lor x\lor d\lor b$  equals $(b\lor d\lor b) \lor x\lor (b\lor d\lor b)$ by the regularity of $\lor$. But the latter reduces to $b\lor x\lor b $ because $b\mathcal D d$, and since $x\in b\suparrow$ it reduces further to  $x$ by Lemma \ref{lemma:costa}, giving $\psi(\varphi (x))=x$.
$\varphi$ must preserve $\land$ and $\lor$. Indeed distributivity gives:
\[\varphi (x\land x') = d\lor (x\land x')\lor d=(d\lor x\lor d)\land (d\lor x'\lor d)=\varphi(x)\land \varphi(x').
\]
And the regularity gives:
\[\varphi(x\lor x') = d\lor (x\lor x') \lor d=(d\lor x \lor d) \lor (d\lor x'\lor d)=\varphi (x) \lor \varphi(x').
\]
Thus $\varphi$ (and $\psi$) is a lattice isomorphism of $b\suparrow$ with $d\suparrow$. But then $\varphi$ and $\psi$ are also isomorphisms of Heyting algebras. That is, e.g., $\varphi(x\rightarrow y)=\varphi(x)\rightarrow \varphi(y)$.

Next, observe that $x\Drel \varphi(x)$  for all $x\in b\suparrow$. Indeed, regularity gives:
\[\varphi(x)\lor x\lor \varphi(x)=(d\lor x\lor d)\lor x\lor (d\lor x\lor d) = d\lor x\lor d = \varphi(x)
\]
and likewise $x\lor \varphi(x)\lor x= \psi(\varphi(x))\lor  \varphi(x)\lor \psi(\varphi(x))=\psi(\varphi(x)) =x$. There is more: $a$ is the unique element in its $\DD$-class belonging to $b\suparrow$ and $c$ is the unique element in the same $\DD$-class belonging to $d\suparrow$ (since each upset $u\suparrow$ intersects any $\DD$-class in at most one element). But  $\varphi(a)$ in $d\suparrow$ behaves in the manner just like $c$, and so $\varphi (a)=c $. Since $b\Drel d$,  $\varphi(b)=d\lor b\lor d=d$  and $\varphi(a\rightarrow b)=\varphi(a)\rightarrow \varphi(b)=c\rightarrow d$, thus giving $a\rightarrow b\Drel c\rightarrow d$.
\end{proof}

Following \cite{bl} a \emph{commutative subset} of a symmetric skew lattice is a non-empty subset
whose elements both join and meet commute.

\begin{thm}\label{theorem:generalized Heyting-skew-Heyting}
Given a co-strongly distributive skew lattice $(S;\land,\lor, 1)$ with top $1$, an operation $\rightarrow$ can be defined on $S$ making  $(S;\land, \lor,\rightarrow, 1)$  a skew Heyting algebra if and only if the operation $\rightarrow$ can be defined on $S/\mathcal D$ making $(S/\mathcal D;\land, \lor,\rightarrow, \mathcal D_1)$ a generalized Heyting algebra.
\end{thm}

\begin{proof}
To begin, for any $u$ in $S$, the upset $u\suparrow$ is a $\DD$-class transversal of the principal filter $S\lor u\lor S$. Secondly, the induced homomorphism $\varphi:S\to S/\DD$  is bijective on any commutative subset of $S$ since distinct commuting elements of $S$ lie in distinct $\DD$-classes. It follows that for each $u$ in $S$, $\varphi$ restricts to an isomorphism of upsets, $\varphi_u: u\suparrow \cong \varphi(u)\suparrow$. Thus each upset $u\suparrow$ in $S$ forms a Heyting algebra if and only if each upset $v\suparrow$ in $S/\mathcal D$, being some $\varphi (u)\suparrow$, must form a Heyting algebra. The theorem follows. 
\end{proof}

\noindent \emph{Comment.} This result is a near-dual of the important fact that a strongly distributive skew lattice $\mathbf S$
with bottom $0$ is the (necessarily unique) reduct of a skew Boolean algebra if and only if its lattice
image $\mathbf S/\DD$ is the reduct of a (necessarily unique) generalized Boolean algebra. (\cite{L4}, {3.8}.)

We next consider  consequences of the above theorem. The first is on the "skew lattice side" of things and the next is more on the "Heyting side". But first recall the definitions of Green's relations $\LL$ and $\RR$ on a skew lattice:
\begin{eqnarray*}
x\LL y \Leftrightarrow (x\land y=x\, \&\, y\land x=y, \text{ or equivalently } x\lor y=y\, \&\, y\lor x=x),  \\
x\RR y \Leftrightarrow (x\land y=y\, \&\, y\land x=x, \text{ or equivalently } x\lor y=x\, \&\, y\lor x=y).
\end{eqnarray*}
Relations $\LL$ and $\RR$ are   contained in the Green's relation $\DD$ defined above  and $\LL\circ \RR=\RR\circ \LL=\DD$ holds. A skew lattice is called \emph{right-handed} if the relation $\LL$ is trivial, in which case $\DD=\RR$, and it is called \emph{left-handed} if the relation $\RR$ is trivial, in which case $\DD=\LL$. By Leech's Second  Decomposition Theorem \cite{L1} the relations $\LL$ and $\RR$ are   congruences on any skew lattice $\mathbf S$, $\mathbf S/\RR$ is left-handed, $\mathbf S/\LL$ is right-handed and the following diagram is a pullback:
 \[\xymatrix{
\mathbf S \ar[r]\ar[d]&\mathbf  S/\RR \ar[d]\\
\mathbf S/\LL  \ar[r] &\mathbf  S/\DD\\
      }
  \]

\begin{cor}\label{new-cor-36}
If $\mathbf S=(S;\land, \lor,1)$ be a co-strongly distributive skew lattice with top $1$, then the following are equivalent:
   \begin{enumerate}
         \item $\mathbf S$ is the reduct of a skew Heyting algebra $(S;\land, \lor, \rightarrow, 1)$.
        \item $\mathbf S/\LL$ is the reduct of a skew Heyting algebra  $(S/\LL;\land, \lor, \rightarrow, 1)$.
            \item $\mathbf S/\RR$ is the reduct of a skew Heyting algebra  $(S/\RR;\land, \lor, \rightarrow, 1)$.
   \end{enumerate}
Moreover, both $\LL$ and $\RR$ are congruences on any skew Heyting algebra.
\end{cor}

\begin{proof}
The equivalence of (i)--(iii) is due to the preceding theorem plus the fact that $\mathbf S/\DD$,  $(\mathbf S/\LL)/\DD_{S/\LL}$ and  $(\mathbf S/\RR)/\DD_{S/\RR}$ are isomorphic. Next, the induced map $\rho:\mathbf S\to \mathbf S/\LL$ is at least a homomorphism of co-strongly distributive skew lattices. By the argument of the preceding theorem, it induces isomorphisms between corresponding pairs of upsets, $u\suparrow$ in $S$ and $\LL_u\suparrow$ in $S/\LL$. Thus given $x\rightarrow y=(y\lor x\lor y)\rightarrow_y y$ and $u\rightarrow v= (v\lor u\lor v) \rightarrow_v v$ with $x;,\LL;, u$ and $y;,\LL;, v$ in $S$, both $(y\lor x\lor y)\rightarrow_y y$ and $(v\lor u\lor v)\rightarrow_v v$ are mapped to the common $\LL_{y\lor x\lor y}\rightarrow_{\LL_y} \LL_y$, making $x\rightarrow y;, \LL;, y\rightarrow v$ in $\mathbf S$. A similar argument applies to the induced map $\lambda: \mathbf S\to \mathbf S/\RR$.
\end{proof}

An alternative to the characterization of Theorem \ref{theorem:axiomatization} is given by:

\begin{cor}\label{cor:def-adjunction}
Every skew Heyting algebra satisfies:
\begin{description}
   \item[(SHA)] $x\preceq y\rightarrow z$ if and only if $x\land y\preceq z$.
\end{description}
In particular, $x\rightarrow y=1$ iff $x\preceq y$.

In general, an algebra  $\mathbf S=(S;\land,\lor,\rightarrow, 1)$ of type $(2,2,2,0)$ is a skew Heyting algebra if the following conditions hold:
\begin{enumerate}
\item The reduct $(S;\land, \lor, 1)$ is a co-strongly distributive skew lattice with top $1$. 
\item $y \leq x\rightarrow y$ holds for all $x,y\in S$.
\item $\mathbf S$ satisfies axiom (SHA).
\end{enumerate}
\end{cor}

\begin{proof}
Given that $S$ is a skew Heyting algebra, since the induced epimorphism $\varphi:S\to S/\DD$ is a homomorphism of skew Heyting algebras we have
\[x\preceq y\rightarrow z \text{ iff } \varphi(x)\leq \varphi(y)\rightarrow \varphi(z) \text{ iff } \varphi(x)\land \varphi(y)\leq \varphi(z) \text{ iff } x\land y\preceq z.
\]
Conversely, let $\mathbf S=(S;\land,\lor,\rightarrow, 1)$  be an algebra of type $(2,2,2,0)$ satisfying (1)--(3). Suppose that $x,y,z$ lie in a common upset $u\uparrow$. Since $\preceq$ is just $\leq $ in $u\uparrow$ nad $y\rightarrow z$ lies in $u\uparrow$ by (2) we have $x\leq y\rightarrow z$ iff $x\land y \leq z$ in $u\uparrow$. $(S;\land,\lor, 1)$ is thus at least a skew Heyting lattice. Now consider the derived implication $\rightarrow^*$ given by $x\rightarrow^* y= (y\lor x\lor y)\rightarrow_y y$. Both $y\rightarrow z$ and $y\rightarrow^* z$ satisfy (SHA) and thus are $\DD$-equivalent. But since both lie in the sublattice $z\suparrow$, they must be equal.
\end{proof}

We have seen that each skew Heyting algebra is basically a co-strongly distributive skew lattice $\mathbf S$
with top, say $1$, for which $\mathbf S/\DD$ is a generalized Heyting algebra, in which case the Heyting structure
on each upset $u\suparrow$ of $\mathbf S$ is obtained from that of the isomorphic upset $\DD_u\suparrow$ in $\mathbf S/\DD$. This suggests that
all standard classes of generalized Heyting algebras yield classes of skew Heyting algebras whose
maximal commutative images belong to the particular class. We consider several cases.

\noindent Case 1. Finite distributive lattices possess a well-defined Heyting algebra structure. Thus any
finite co-strongly distributive skew lattice with a top, or more generally any co-strongly distributive
skew lattice with a top and a finite maximal lattice image is the reduct of a unique skew Heyting
algebra.

\noindent Case 2. All chains possessing a top 1 form Heyting algebras. Here things are simple:
\[x\rightarrow y =\left\{
\begin{array}{ll}
1; & \text{if } x\leq y.  \\
y; &\text{otherwise}.
\end{array}
\right.
\]
Thus all co-strongly distributive skew chains with a top are skew Heyting algebra reducts, where a \emph{skew chain} is any skew lattice $\mathbf S$ where $\mathbf S/\DD$ is a chain, i.e., $\preceq$ is a total preorder on $\mathbf S$. Here, given $x, y$ in a common $u\suparrow$ one has:
\[x\rightarrow y =\left\{
\begin{array}{ll}
1; & \text{if }x\preceq y.  \\
y; &\text{otherwise}.
\end{array}
\right.
\]

\noindent Case 3. Dual generalized Boolean algebras. These are algebras $\mathbf S=(S;\land,\lor,\dd,1)$ where $(S;\land,\lor, 1)$ is a distributive lattice with top $1$ and $\dd$ is a binary operation on $S$ such that $(y\lor x)\lor (y \dd x) =1$ and $(y\lor x)\land (y\dd x)=y$ for all $x,y$ in $S$. As with $\setminus$ for generalized Boolean algebras, $\dd$ is uniquely determined. Moreover, in this case, $x\rightarrow y=y\dd x$. A dual-skew Boolean algebra $\mathbf S=(S;\land, \lor, \dd ,1)$ is an algebra such that $(S;\land, \lor, 1) $ is a co-strongly distributive skew lattice with top $1$ and binary operation $\dd$ such that:
\begin{eqnarray*}
(y\lor x\lor y)\lor (y\dd x) = 1= (y\dd x)\lor (y\lor x\lor y); \\
(y\lor x\lor y)\land (y\dd x) = y= (y\dd x)\land (y\lor x\lor y).
\end{eqnarray*}
The relevant diagram is:
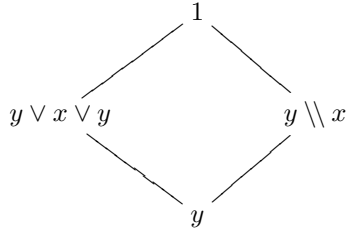
\[
\xymatrix{
 & 1\ar@{-}[dl] \ar@{-}[dr]  &  \\
 y\lor x\lor y\ar@{-}[dr]& & y\dd x\ar@{-}[dl] \\
 & y & 
   }
\]
These dual algebras are, of course, precisely the co-strongly distributive skew lattices with a top whose maximal lattice images are the lattice reducts of dual-generalized Boolean algebras. Once again we follow the commutative case: $x\rightarrow y=y\dd x$ which now is $y\dd (y\lor x\lor y)$ in $y\suparrow $.

We thus have:

\begin{cor}\label{cor37}
A co-strongly distributive skew lattice with a top $\mathbf S=(S;\land,\lor, 1)$ is the reduct of a uniquely determined skew Heyting algebra $(S;\land,\lor, \dd, 1)$ if any one of the following conditions holds:
\begin{enumerate}
   \item $\mathbf S/\DD$ is finite.
  \item $\mathbf S$ is a skew chain.
   \item $\mathbf S$ is the reduct of a dual generalized Boolean algebra, $\mathbf S=(S;\land, \lor,\dd, 1)$.
\end{enumerate}
\end{cor}

Implicit in Case 3 is a basic duality that occurs for skew lattices. Given a skew lattice $\mathbf S=(S;\land, \lor)$, its (vertical) \emph{dual} is the skew lattice $\mathbf S^\updownarrow=(S;\land^\updownarrow, \lor^\updownarrow)$, where as binary functions, $\land^\updownarrow=\lor$ and $\lor^\updownarrow=\land$. Clearly $\mathbf S^{\updownarrow \updownarrow}=\mathbf S$ and any homomorphism $f:\mathbf S\to\mathbf  T$ of skew lattices ia also a homomorphism from $\mathbf S^\updownarrow$ to $\mathbf T^\updownarrow$; moreover a skew lattice $\mathbf S$ is distributive (or symmetric) iff $\mathbf S^\updownarrow$ is thus. Either $\mathbf S$ or $\mathbf S^\updownarrow$ is strongly distributive iff the other is co-strongly distributive; more generally, $\mathbf S$ or $\mathbf S^\updownarrow$ is normal iff the other is co-normal. Also, one has a bottom element iff the other has a top element, both being the same element in $S$.

Expanding the signature, $(S;\land, \lor,\setminus, 0) $ is a skew Boolean algebra if and only if its dual  $(S;\land^\updownarrow, \lor^\updownarrow,\dd, 1) $ is a dual skew Boolean algebra where $\setminus$ and $0$ are replaced by $\dd$ and $1$. Thus any skew Boolean algebra $(S;\land, \lor, \setminus, 0)$ induces a skew Heyting algebra $(S;\land^\updownarrow, \lor^\updownarrow,\rightarrow, 1) $ where $x\rightarrow y=y\setminus x$ and $1=\text{old } 0$. Standard examples of skew Boolean algebras thus give us:

\begin{exmp}
Given sets $X$ and $Y$, the skew Heyting operations derived from the skew Boolean operations on the set $\mathcal P(X,Y)$ of all partial functions from $X$ to $Y$ are as follows. 
\begin{center}
  \begin{tabular}{| c | c | c| }
\hline
    skew Heyting operation & description & skew Boolean operation \\ \hline
    $f\land g$ & $f\cup (g|_{\dom g-\dom f})$ & $f\lor g$ \\
 $f\lor g$ & $g|_{\dom g\cap \dom f}$ & $f\land g$ \\ 
 $f\rightarrow g$ & $g|_{\dom g-\dom f}$ & $g\setminus f$ \\ 
    $1$ & $\emptyset$  & $0$ \\
\hline
  \end{tabular}
\end{center}
\end{exmp}

\begin{exmp}
Given a surjective function $\pi: Y \to X$, let $Sec(\pi)$ denote the set of
all \emph{sections} of $\pi$, that is, functions $f$ from subsets $U$ of $X$ to $Y$ such that $\pi \circ f =
id_{\dom (f)}$. Skew Heyting algebra operations and corresponding skew Boolean
operations are defined on $Sec(\pi)$ using precisely the above descriptions. At first
glance this seems to be just a subalgebra of the type of algebra in Example 1. The
latter, however, is isomorphic to  $Sec(\pi)$ where $\pi$  is now the coordinate projection
of $ X\times Y$ onto $X$. Modifications of this example arise in the dualities of the next
section.

It so happens that any right-handed (co-)strongly distributive skew lattice is
isomorphic to a subset of partial functions in some $\mathcal  P(X, Y)$ that is closed under the
relevant $\land $ and $\lor $ operations above. (See [12] Section 3.7.)  It follows that the
skew lattice reduct of a skew Heyting algebra is isomorphic to some such partial
function algebra. The difference of this more general case from that of the two
examples above is that here $x\rightarrow y$ need not have a polynomial definition, unlike the
co-Boolean case where $x\rightarrow y = y\setminus x$.
\end{exmp}

%It follows from Corollary \ref{cor:def-adjunction} that $x\rightarrow y =1$ if and only if $x\preceq y$. The skew Heyting implication thus determines  a reflexive and transitive relation on $S$ that is antisymmetric exactly in the commutative case.

The following result is useful for computing in skew Heyting algebras.

\begin{prop}\label{prop:imp-or}
Let $\mathbf S=(S;\land,\lor,\rightarrow, 1)$ be a skew Heyting algebra and $x,y,z\in S.$ Then
\[(x\lor y\lor x)\rightarrow z=(x\rightarrow z)\land (y\rightarrow z)\land (x\rightarrow z).
\]
\end{prop}

\begin{proof}
As $\mathbf S/\DD$ is a generalized Heyting algebra and relation $\DD$ respects all skew Heyting algebra operations, it follows that 
$(x\lor y\lor x)\rightarrow z \,\DD\, (x\rightarrow z)\land (y\rightarrow z)\land (x\rightarrow z).
$
However, both $(x\lor y\lor x)\rightarrow z$ and $(x\rightarrow z)\land (y\rightarrow z)\land (x\rightarrow z)$ are above $z$ with respect to the natural partial order, and hence must be equal by Lemma \ref{lemma:conormal}.
\end{proof}

 A skew lattice $\mathbf S$ is \emph{meet [join]
complete} if each nonempty commutative subset possesses an infimum [a supremum] in $S$. It follows from the dual of \cite{bl} Proposition 2.10 that if $\mathbf S$ is a meet complete co-strongly distributive  skew lattice with $1$, then $\mathbf S$ is complete.
We call a skew Heyting algebra  \emph{complete} if it is complete as a skew lattice.

\section{Connections to duality}

 Dual skew Boolean algebras are order duals (upside-downs) to usually studied skew Boolean algebras. Skew Boolean algebras and dual skew Boolean algebras are categorically isomorphic. Right-handed (dual) skew Boolean algebras are dually equivalent to sheaves over locally compact Boolean spaces by results of \cite{BCV} and \cite{K}, where a \emph{locally compact Boolean space} is a topological space whose one-point-compactification is a Boolean space.  
The obtained duality asserts that any right- [left-]handed skew Boolean algebra is isomorphic to the skew Boolean algebra of compact open sections (i.e. sections over compact open subsets) of an \'etale map over some locally compact Boolean space. Let us  note that the restriction to right- [left-]handed algebras is not a major restriction since  Leech's Second Decomposition Theorem yields that any skew lattice is a pull back of a left-handed and a right-handed skew lattice over their common maximal lattice image \cite{L1}. The general two-sided case was also covered in \cite{BCV}.

Bounded distributive lattices are dual  to Priestley spaces; in this duality each bounded distributive lattice is represented as the distributive lattice of all clopen (i.e. closed and open) upsets of a Priestley space.
The Esakia duality established in \cite{esakia} yields that Heyting algebras are dual to \emph{Esakia spaces}, i.e. those Priestley spaces in which the downset of each clopen set is again clopen.  
 Moreover, if $(X, \leq, \tau)$ is an Esakia space then given clopen subsets $U$ and $V$ in $X$ the implication is defined by
\[U\rightarrow V=X\setminus \duparrow (U\setminus V).
\]

Duality for strongly distributive skew lattices was recently established in \cite{skew-priestley}. It yields that right-handed strongly distributive skew lattices with bottom  are dual  to the sheaves over  locally Priestley spaces, where by a \emph{locally Priestley space} we mean an ordered topological space whose one-point-compactification is a Priestley space. Via the obtained duality each right-handed strongly distributive skew lattice with bottom is represented as a skew lattice of sections over copen (i.e. compact and open) downsets of a locally Priestley space, with the operations being defined as follows:
\begin{align*}
  0 &= \emptyset,\\
  r \land  s &= \restricted{s}{\dom{r} \cap \dom{s}},\\
  r \lor s &= r \cup \restricted{s}{\dom{s} - \dom{r}},\\
  r \setminus s &= \restricted{r}{\dom{r} - \dom{s}}.
\end{align*}

Given a distributive lattice $\mathbf L$ denote by $\mathbf L^c$ the distributive lattice that is obtained from $\mathbf L$ by reversing the order. Denote by $\mathbf{DL}$ the category of all distributive lattices, by $\mathbf{PS}$ the category of all locally Priestley spaces and consider the functors:
\[\begin{array}{rccl}
^c: & \mathbf{DL} & \to & \mathbf{DL} \\
&\mathbf  L & \mapsto &\mathbf  L^c
\end{array} \qquad
\text{and}
\qquad
\begin{array}{rccl}
^r: & \mathbf{PS} & \to & \mathbf{PS} \\
& (X,\leq) & \mapsto & (X, \geq).
\end{array}
\]
Restricting the functors $^c$ and $r$ to the categories $\mathbf{HA}$ of all Heyting algebras
and $\mathbf{ES}$ of all Esakia spaces, respectively, yields the following isomorphism of categories:
\[\begin{array}{rccl}
^c: & \mathbf{HA} & \to & \mathbf{cHA} \\
& \mathbf L & \mapsto &\mathbf  L^c
\end{array} \qquad
\text{and}
\qquad
\begin{array}{rccl}
^r: & \mathbf{ES} & \to & \mathbf{cES} \\
& (X,\leq) & \mapsto & (X, \geq),
\end{array}
\]
where $\mathbf{cHA}$ denotes the category of all \emph{co-Heyting algebras} (defined as order-inverses of Heyting algebras) and $\mathbf{cES}$ denotes the category of all 
\emph{co-Esakia spaces} the latter being introduced in \cite{bitop} as Priestley spaces in which an upset of a clopen is again clopen.

We introduce the following categories:
\[
\begin{array}{l}
\mathbf{SDSL}: \text{ strongly distributive skew lattices with $0$},\\
\mathbf{cSDSL}: \text{ co-strongly distributive skew lattices with $1$},\\
\mathbf{SHA}: \text{ skew Heyting algebras},\\
\mathbf{cSHA}: \text{ co-skew Heyting algebras},
\end{array}
\]
with the latter being defined as the category of all algebras of the form $\mathbf S^c$, where $\mathbf S$ is a skew Heyting algebra and 
\[\begin{array}{rccl}
^c: & \mathbf{SDSL} & \to & \mathbf{cSDSL} \\
& \mathbf S & \mapsto &\mathbf  S^c
\end{array} 
\]
is the isomorphism of categories that turns a skew lattice upside-down. The restriction of the functor $^c$ to the categories  $\mathbf{cSHA}$  and $\mathbf{SHA}$ yields the isomorphism:
\[\begin{array}{rccl}
^c: & \mathbf{cSHA} & \to & \mathbf{SHA} \\
&\mathbf  S & \mapsto &\mathbf  S^c
\end{array} .
\]
The isomorphism of categories induces an isomorphism of concepts:
\[
\begin{tabular}{|c|c|}
\hline
$\mathbf{SHA}$ & $\mathbf{cSHA}$ \\
\hline
$\land$ & $\lor$ \\
$\lor$ & $\land$ \\
$1$ & $0$ \\
strong codistributivity & strong distributivity \\
filter & ideal \\
prime filter & prime ideal \\
\hline
\end{tabular}
\]

It follows from  Theorem \ref{theorem:generalized Heyting-skew-Heyting} that the skew Heyting algebra structure can be imposed exactly on those co-strongly distributive skew lattices with top whose maximal lattice image is a generalized Heyting algebra.   Therefore the duality for right-handed skew Heyting algebras  yields that they are dual to sheaves over \emph{local Esakia spaces}, i.e. ordered topological spaces whose one-point-compactification is an Esakia space.

Let $(B, \leq)$ be an Esakia space, $E$ a topological space and $p:E\to B$ a surjective \'etale map. Consider the set $S$ of all  sections of $p$ over copen upsets in $B$, i.e.  an element of $S$ is a map $s:U\to E$, where  $U$ is a copen upset in $B$, that satisfies the property $p\circ s= \id_U$. A section $s\in S$ is considered to be below a section $r\in S$ when $s$ extends $r$. 
The skew Heyting operations are defined on $S$ by:
\begin{eqnarray*} 
r \lor s = \restricted{s}{\dom{r}\cap \dom{s}}, \\
r\land s = r\cup \restricted{s}{\dom{s}\setminus \dom{r}}, \\
r\rightarrow s = \restricted{r}{\suparrow(\dom{s}\setminus \dom{r})}\\
1 = \emptyset.
\end{eqnarray*}
\begin{thm}
Let $p:E\to B$ be a surjective \'etale map over a local Esakia space $B$. Then the set $S$ of all sections of $p$ over copen upsets in $B$ forms a skew Heyting algebra under the above operations.
\end{thm}
%

%% Bibliography
%Examples of references:~\cite{BI1}, \cite{BI2}, \cite{BI3}.

\end{document}